\documentclass{amsart}
\usepackage[colorlinks,pdfpagemode=none,pdfstartview={XYZ null null null}]{hyperref}
\headheight\baselineskip
\headsep\baselineskip
\advance\topmargin-\headheight % but shift the headings into the margin
\advance\topmargin-\headsep

\newcommand{\Z}{\mathcal{Z}}            % sequences
\newcommand{\W}{\mathcal{W}}            % sequences

\newcommand{\I}{\mathcal{I}}            %
\newcommand{\J}{\mathcal{J}}            %
\newcommand{\N}{\mathcal{N}}            % subspace vanishing on \Z
\newcommand{\ID}{\mathbb{D}}            % unit disk
\newcommand{\IC}{\mathbb{C}}            % complex plane
\newcommand{\dbar}{\bar\partial}
\newcommand{\laplace}{\partial\dbar}
\newcommand{\invL}{\mathop{\tilde\Delta}}
\newcommand{\grad}{\nabla}
\newcommand{\term}{\emph}
\newcommand{\st}{:}
\newcommand{\re}{\mathop{\mathrm{Re}}\nolimits}
\newcommand{\avg}{\mathop{\mathrm{avg}}\nolimits}
\newcommand{\defeq}{\stackrel{\mathrm{def}}=}
\let\eps\epsilon
\let\intersect\cap

\let\union\cup
\let\Union\bigcup
\let\phi\varphi
\let\term\emph
\newtheorem{theorem}{Theorem}[section]
\newtheorem{lemma}[theorem]{Lemma}
\newtheorem{proposition}[theorem]{Proposition}

\newtheorem{Athm}{Theorem}

\theoremstyle{remark}
\newtheorem{remark}{Remark}
\newtheorem{example}{Example}

\numberwithin{equation}{section}

\begin{document}

%topmatter
\title[Interpolation without Separation] %
{Interpolation without Separation in Bergman Spaces}
\author{Daniel H. Luecking}
\address{Department of Mathematical Sciences\\
         University of Arkansas\\
         Fayetteville, Arkansas 72701}
\email{luecking@uark.edu}
\date{Sept. 9, 2004\thanks{This paper was indeed written in 2004, the
results having been presented at the AMS meeting in Nashville, TN, that
year. It languished on my personal web page until I uploaded it to arxiv
in 2014. Almost immediately thereafter, I received an email pointing me
to some related work. Section~\ref{sec:addendum} was added in response
to that; apart from minor corrections, that is the only addition since
2004.}}


\begin{abstract}
    Most characterizations of interpolating sequences for Bergman spaces
    include the condition that the sequence be uniformly discrete in the
    hyperbolic metric. We show that if the notion of interpolation is
    suitably generalized, two of these characterizations remain valid
    without that condition. The general interpolation we consider
    here includes the usual simple interpolation and multiple
    interpolation as special cases.
\end{abstract}

\maketitle

\section{Introduction}\label{sec:introduction}

Let $dA$ denote area measure and let $G$ be a domain in the complex
plain. Let $L^p(G)= L^p(G,dA)$ be the usual Lebesgue
space of measurable functions $f$ with $\| f \|_{p,G}^p = \int_G |f|^p
\,dA < \infty$. The \term{Bergman space $A^p(G)$} is the closed
subspace of $L^p(G)$ consisting of analytic functions. If $1 \le p <
\infty$, these are Banach spaces and if $0<p<1$ they are quasi-Banach
spaces. We will allow all $0<p<\infty$ and abuse the terminology by
calling $\| \cdot \|_{p,G}$ a norm even when $p < 1$.  In the case where
$G=\ID$, the open unit disk, we will abbreviate: $L^p = L^p(\ID)$,
$A^p= A^p(\ID)$~and $\| \cdot \|_p = \| \cdot \|_{p,\ID}$.

Let $\psi(z,\zeta)$ denote the \term{pseudohyperbolic metric}:
\begin{equation*}
  \psi(z,\zeta) = \left|\frac{z-\zeta}{1-\bar\zeta z}\right|.
\end{equation*}
We will use $D(z,r)$ for the \term{pseudohyperbolic disk} of radius $r$
centered at $z$, that is, the ball of radius $r<1$ in the
pseudohyperbolic metric. Let $\rho(z) = 1 - |z|$, denote the distance of
a point to the boundary of $\ID$ and, for a set $S$ let $\rho(S) =
\max\{ \rho(z) \st z\in S \}$. Moreover, let $d\lambda(z) = (1 -
|z|^2)^{-2}dA(z)$ denote the \term{invariant area measure} on $\ID$.

We use $|S|$ to denote the Euclidean area of a set $S$ and we recall
that $|D(z,r)|$ behaves like $(1 - |z|)^2$ as $|z| \to 1$ for fixed $r <
1$, and like $r^2$ as $r\to 0$ for fixed $z$. We also recall that
$\rho(D(z,r))$ behaves like $1 - |z|$ for fixed $r$.

Let $\Z= \{ z_k : k=1,2,3,\dots \}$ be a sequence in $\ID$ without limit
points in $\ID$. The sequence $\Z$ is said to be \term{uniformly
discrete} if there is a lower bound on the pseudohyperbolic distance
between entries. That is, there is an $\epsilon > 0$ such that
\begin{equation*}
     \psi(z_k,z_n) > \epsilon \quad \text{ \ for all $k\ne n$}.
\end{equation*}
A uniformly discrete sequence necessarily contains no repeated values.

The usual $A^p$ interpolation problem for $\Z$ is the following: given a
sequence $w = (w_k)$ of complex numbers satisfying $\sum_{j=1}^\infty
|w_j|^p (1 - |z_j|^2)^2 < \infty$, find a function $f\in A^p$ such that
$f(z_k) = w_k$ for all $k$. The sequence $\Z$ is called an
\term{interpolating sequence for $A^p$} if every such interpolation
problem has a solution. It is well known that interpolating sequences
have to be uniformly discrete. K.~Seip characterized interpolating
sequences for $A^p$ as those that are uniformly discrete and in addition
satisfy a certain bound on the upper uniform density (to be defined in
section~\ref{sec:density}): $D^+(\Z) < 1/p$.

M. Krosky and A. Schuster \cite{KS01} investigated multiple
interpolation, in which not only the values of $f$ at $z_k$ but also the
values of its derivatives at $z_k$ can be prescribed. The multiple
interpolation problem is the following: given a sequence $w = (w_k)$ of
$n$-tuples $w_k = (w_k^{(0)}, w_k^{(1)}, \dots, w_k^{(n-1)})$ satisfying
$\sum_{k=1}^\infty \sum_{l=0}^{n-1} |w_k^{(l)}|^p (1 - |z_k|^2)^{lp+2}$,
find $f\in A^p$ such that $f^{(l)}(z_k) = w_k^{(l)}$ for all $0\le l\le
n-1$ and all $k\ge 1$. A sequence $\Z$ is called a \term{multiple
interpolating sequence of order $n$ for $A^p$} if this interpolation
problem has a solution for every such sequence $w$.

Krosky and Schuster showed that multiple interpolating sequences $\Z$
are characterized by being uniformly discrete and satisfying the density
bound $D^+(\Z) < 1/(np)$. However, if we encode the multiplicity of the
interpolation by repeating each distinct point $z_k$ in $\Z$ the
appropriate number of times, then the density condition is the same as
Seip's: $D^+(\Z) < 1/p$. The resulting sequence $\Z$ is not uniformly
discrete, only the sequence of distinct values is.

In this paper we will to remove even this requirement that $\Z$ have
uniformly discrete values. We will formulate a more general interpolation
problem in $A^p$, and show that any sequence (not necessarily uniformly discrete)
is an interpolating sequence for this problem if and only if
$D^+(\Z) < 1/p$. This will cover the case of multiple interpolation at
repeated points, but it will also include consideration of distinct
points without a positive lower bound on $\psi(z_k,z_n)$.

In \cite{Lue04a} this author showed that the usual interpolating
sequences could also be characterized in terms of a $\dbar$-problem in a
suitable weighted $L^p$ space. The function
\begin{equation*}
    k_\Z(\zeta) = \sum_{a\in\Z} k_{a}(\zeta),
    \quad\text{with}\quad k_a(\zeta) =
    \frac{|\zeta|^2}{2}\frac{(1-|a|^2)^2}{|1 - \bar a z|^2}
\end{equation*}
completely characterizes zero sequences for $A^p$ (\cite{Lue96}) in the
following way: $\Z$ is a zero sequence for $A^p$ if and only if the
weighted space $L^p_\Z = L^p(\ID, \exp(pk_\Z) \,dA)$ contains nonzero
analytic functions. We pose the following weighted $\dbar$-problem:
given $f\in L^p_\Z$, find a function $u \in L^p_\Z$ satisfying
\begin{equation}\label{eq:dbar}
    (1 - |z|^2)\dbar u(z) = f(z)
\end{equation}
in $\ID$. We will say that \term{the $\dbar$-problem has solutions
with bounds in $L^p_\Z$} if there is a constant $C$ such that for any
$f\in L^p_\Z$ there is a solution $u$ to \eqref{eq:dbar} in $L^p_\Z$
satisfying $\| u \|_{p,\Z} \le C\| f \|_{p,\Z}$. The notation $\| \cdot
\|_{p,\Z}$ denotes the obvious norm for $L^p_\Z$.

The main result of \cite{Lue04a} is that $\Z$ is an interpolating
sequence for $A^p$ if and only if it is uniformly discrete and the
$\dbar$-problem has solutions with bounds in $L^p_\Z$. In case $p < 1$
one has to modify the space $L^p_\Z$ so that its elements are locally
integrable in a suitable way, but otherwise the result is valid for all
$0 < p < \infty$.

In this paper we will show that for the newly formulated general
interpolation problem, any sequence having bounded density is
interpolating for $A^p$ if and only if the $\dbar$-problem has solutions
with bounds in $L^p_\Z$. A sequence has bounded density if for a given
$0 < R < 1$ the cardinality of the set of $k$ with $z_k\in D(z,R)$ has
an upper bound independent of $z$.

In the next section we will formulate the general interpolation
problem for $A^p$ and define general interpolating sequences for
$A^p$ to be those where all such problems have a solution.

In section~\ref{sec:properties} we will prove some basic properties of
general interpolating sequences for $A^p$. This will include two
necessary conditions that are so basic to interpolation that we will
call sequences that satisfy them admissible sequences. Then in
section~\ref{sec:admissible} we show that, in fact, sequences are
admissible if and only if they have bounded density.

In section~\ref{sec:density} we will show that if a sequence $\Z$ is a
general interpolating sequences for $A^p$, then it satisfies the
mentioned upper density inequality $D^+(\Z) < 1/p$. We will complete the
proof in section~\ref{sec:interpolating}, showing that the density
condition allows us to construct a solution operator for the
$\dbar$-equation, and solutions of the $\dbar$-equation allow us to
construct solutions of any interpolation problem. Our proof that the
density condition implies solutions of the $\dbar$-equation is more
direct and simpler than our proof of a similar implication in
\cite{Lue04a}. The proof does not appeal to J.~Ortega-Cerd\`a's result on
weighted $\dbar$-problems (\cite{Ort02}) but does make use of properties
of the weight $e^{pk_\Z}$ that Ortega-Cerd\`a's work does not assume.
Apart from this, once the main properties of general interpolating
sequences are established, the arguments of sections
\ref{sec:density}~and \ref{sec:interpolating} do not differ much from
those in \cite{Lue04a}.

In the final section we supply a few examples of general
interpolation problems and the associated interpolating sequences. We
will also observe that the methods apply without significant change to
the usual weighted Bergman spaces (with the $L^p(dA_\alpha)$ norm, where
$dA_\alpha(z) = (1 - |z|^2)^\alpha\,dA(z)$, $\alpha > -1$). The
applicable density condition is then $D^+(\Z) < (\alpha + 1)/p$ and the
appropriate space for the $\dbar$-problem is
$L^p(\ID,\exp(pk_\Z)\,dA_\alpha)$

\section{The interpolation process}\label{sec:scheme}

Before formulating the interpolation process, we answer the obvious
question: Why are interpolating sequences ``always'' required to be
uniformly discrete? We answer this question via the following
proposition. Similar results have been rather vaguely stated in several
places, but we have found no convenient reference in the following
generality. Let us define $M_a(z) = (a-z)/(1 - \bar a z)$, the
\term{Moebius transformation} of $\ID$ that takes $a$ to $0$ (and $0$
to $a$). Let $e^{(k)}$ denote the sequence having a $1$ in position $k$
and $0$ elsewhere. Let $P_k$ be the operator of projection onto the
$k$th component: if $w = (w_j)$ then $P_k(w) = w_k e^{(k)}$

\begin{proposition}\label{thm:discreteness}
  Let $A$ be a quasi-Banach space of analytic functions on $\ID$ and $\Z
  = \{ z_k : k= 1,2,\dots \}$ a sequence of distinct points in $\ID$.
  Let $X$ a quasi-Banach sequence space, the sequences being indexed the
  same as $\Z$. Assume the following:
  \begin{enumerate}
    \item For every $k$ the sequences $e^{(k)}$ belongs to $X$.
    \item For every $k$, $P_k$ takes $X$ continuously into $X$ and
        $\sup_k \| P_k \| < \infty$.
    \item There is a constant $C_A$ such that for any $k$, if
        $f\in A$ and $f(z_k) = 0$, then $f/M_{z_k} \in A$ and $\|
        f/M_{z_k} \|_A \le C_A\| f \|_A$.
  \end{enumerate}
  If the operator $\Phi$ taking $A$ to sequences via $\Phi(f)_k = f(z_k)$
  is continuous and onto $X$, then $\Z$ is uniformly discrete.
\end{proposition}

\begin{proof}
By the open mapping principle there is a constant $K$ such that for every
sequence $w\in X$ there is a function $f \in A$ such that $\Phi(f) = w$
and $\| f \|_A \le K\| w \|_X$.

Let $z_k, z_n \in \Z$, $k \ne n$. The assumptions imply that the
sequence $w = M_{z_n}(z_k)e^{(k)}$ belongs to $X$. Let $f \in A$ with
$\Phi(f) = w$ and $\| f \|_A \le K\| w \|_X$. Since $f$ vanishes at
$z_n$ we have $g = f/M_{z_n} \in A$. Note that $g(z_k) = 1$, and so
$P_k(\Phi(g)) = e^{(k)}$. Then we have the following
inequalities:
\begin{equation*}
    \| e^{(k)} \| \le C_1\| \Phi(g) \| \le C_2 \| g \| \le
    C_3 \| f \| \le C_4 \| w \| = C_4  |M_{z_n}(z_k)| \| e^{(k)} \|,
\end{equation*}
where $C_1 = \sup_k\| P_k \|$, $C_2 = C_1\| \Phi \|$, $C_3 = C_A C_2$,
and $C_4 = KC_3$. We immediately obtain $\psi(z_k,z_n) = |M_{z_n}(z_k)|
\ge 1/C_4$.
\end{proof}

So the keys to uniform discreteness are the lattice properties of the
usual sequence spaces and the ability to divide out zeros in a uniform
way in the usual function spaces. We note that the sequence space $X$ in
the previous proof could be vector-valued with essentially the same
result. For example, the multiple interpolation problem is governed by
the above proposition. Partly for this reason it seems impossible to do
away completely with some sort of separation condition (see
Theorem~\ref{thm:lowerbound})

Now consider the multiple interpolation problem (of which the simple
interpolation problem is the special case $n=1$) from our introduction.
In that problem we have associated to each distinct point $z_k$ in $\Z$
a finite dimensional space, namely $\IC^n$. There is also an associated
linear \term{interpolation operator} taking analytic functions $f$ into
that space, namely the evaluation of $f$~and its first $n-1$ derivatives
at $z_k$.

An equivalent way to do multiple interpolation is to replace $\IC^n$
with a quotient space: all functions analytic in a neighborhood $G_k$ of
$z_k$ divided by the ideal of all such functions that have a zero of
order at least $n$ at $z_k$. And replace the operator of evaluating $f$
and its derivatives at $z_k$ with that of restricting $f$ to $G_k$ and
applying the quotient mapping. One need only determine a suitable norm
on this quotient space and a natural choice (given that we are
investigating interpolation in $A^p$) would be the quotient norm in
$A^p(G_k)$. This requires us to make a choice of the neighborhood $G_k$.
It is relatively straightforward (see section~\ref{sec:examples}) to
verify that a choice producing a norm equivalent to the norm used in the
introduction is to let each $G_k$ be a pseudohyperbolic disk $D(z_k, R)$
for some fixed radius $R<1$.

Once it has been shown that a multiple interpolating sequence must be
uniformly discrete, it is possible to replace the radius $R$ chosen with
one making the $G_k$ disjoint, but this is by no means necessary for
defining the quotient space and its norm.

For our interpolation process we will \emph{start} with a sequence of
domains $G_k$. If we let $R_k$ denote the diameter of $G_k$ in the
pseudohyperbolic metric, we will require an upper bound $\sup_k R_k = R
< 1$. We also require, of course, that $\Z\subset \bigcup_k G_k$, but do
not require that the $G_k$ be disjoint. Then we partition the sequence
$\Z$ into disjoint finite subsequences $\Z_k$ where each $\Z_k$ belongs
to $G_k$. We will refer to the $\Z_k$ as \term{clusters}. We require
that a repeated point of $\Z$ has all its repetitions in the same
cluster. Finally, we require that $G_k$~and $\Z_k$ be chosen so that
$(\Z_k)_\epsilon \subset G_k$ for some $\epsilon > 0$ independent of
$k$, where the notation $(S)_\epsilon$ denotes the
$\epsilon$-neighborhood of a set $S$ in the pseudohyperbolic metric.

As we have seen, evaluating a function and its derivatives at a point is
equivalent to restricting the function to a neighborhood and applying a
quotient mapping. In this more general setting, we associate with each
cluster $\Z_k$ the subspace $\N_k= \N(\Z_k,G_k)$ of $A^p(G_k)$ consisting
of all functions whose zero set contains $Z_k$ (counting multiplicity)
and the corresponding finite dimensional dimensional quotient space $E_k
= E(\Z_k, G_k)= A^p(G_k)/\N_k$. The norm $\| \cdot \|_{E_k}$ on $E_k$ is
the quotient norm, that is, for $w_k\in E_k$
\begin{equation*}
  \| w_k \|_{E_k}^p = \inf \int_{G_k} |g|^p \,dA
\end{equation*}
where the infimum is taken over all functions $g$ in the equivalence
class $w_k$. A function $f$ interpolates $w_k$ if $f|_{G_k}$ belongs to
the equivalence class $w_k$.

We could equivalently take $E_k$ to be $\IC^{n_k}$ (where $n_k$ is the
cardinality of $\Z_k$ counting repetitions), and take the operator to be
evaluation of functions (and derivatives) at points of $\Z_k$, but we
would still define the norm of the $n_k$-tuple $w_k$ in the same way,
taking the infimum over all $g$ whose values~and derivatives produce
$w_k$.

One can do interpolation with almost any choice of finite dimensionsal
spaces $E_k$, together with a choice of norms on those spaces and a
choice of interpolation operator taking functions in $A^p$ into these
spaces.

In this paper, we will always define the spaces as above, always define
the norms as above, and always define the interpolation operator as
above. Thus, our interpolation process will depend only on the choice of
open sets $G_k$ and the choice of subdivision into clusters $\Z_k$. We
define the term \term{interpolation scheme} to mean a particular choice
of domains and a particular subdivision into clusters satisfying the
following two requirements:
\begin{enumerate}
  \item there exists $\eps>0$ such that $(\Z_k)_\eps \subset G_k$ for
        every $k$, and
  \item there exists $0<R<1$ such that for every $k$ the
        pseudohyperbolic diameter of $G_k$ is no more than $R$.
\end{enumerate}

It is convenient to assign a name such as $\I$ to an interpolation
scheme, especially when more than one such scheme is under discussion.
We will usually denote by $E_k$ the finite dimensional space defined and
normed as above and call them the \term{target spaces} of the scheme
$\I$. We will refer to the $\Z_k$ as the \term{clusters} of $\I$, the
$G_k$ as the \term{domains} of $\I$, the upper bound $R$ as the
\term{diameter} of $\I$ and the largest $\eps$ with $(\Z_k)_\eps \subset
G_k$ as the \term{inner radius} of $\I$.

Given an interpolation scheme $\I$, the \term{general interpolation
problem for $A^p$} is the following: given a sequence $(w_k)$ with $w_k
\in E_k$ for each $k$, satisfying $\sum \| w_k \|_{E_k}^p < \infty$,
find a function $f \in A^p$ such that $f|_{G_k}$ belongs to the
equivalence class $w_k$ for every $k$. We call $\Z$ an
interpolating sequence for $A^p$ \term{relative to the scheme $\I$} if every
such problem has a solution.

A sequence $\Z$ can be general interpolating for some choices of
$\I$ and not for others. However, we will soon see that if a sequence is
interpolating relative to a scheme $\I$, then two additional conditions
will have to be satisfied. We will call schemes that satisfy these
necessary conditions ``admissible'' and then it will turn out that a
sequence is interpolating relative to one admissible scheme if and only
if it is interpolating relative to any other admissible scheme. Hence we
will sometimes omit the modifier ``relative to to the scheme $\I$''.

We will show that a general interpolating sequence requires that there
be a positive lower bound on the pseudohyperbolic distance between
distinct clusters, and a finite upper bound on the number of points in
each cluster. These two together imply that the sequence $\Z$ has
bounded density. Conversely, we will see that for any sequence with
bounded density there is an admissible interpolation scheme. We will
also describe a simple way to select an admissible scheme.

We point out now that the purpose of the condition $(\Z_k)_\epsilon
\subset G_k$ is to ensure that the norm assigned to $E_k$ is reasonable.
Without such a lower bound on the distance to the boundary of $G_k$,
there can be functions in $A^p(G_k)$ with abitrarily small norm but with
arbitrarily large values. While this could be a perfectly valid problem,
it would make interpolation by functions in $A^p$ impossible.

Given an interpolation scheme $\I$, let us denote by $X^p_\I$ the $l^p$
direct sum of the $E_k$. That is, all sequences $w = (w_k)$ with $w_k
\in E_k$ satisfying
\begin{equation*}
    \| w \|_{X^p_\I}^p = \sum_{k=1}^\infty \| w_k \|_{E_k}^p < \infty.
\end{equation*}
Defining a sequence in $X^p_\I$ is equivalent to selecting functions
$f_k$ analytic in each $G_k$ with
$$
    \sum_{k=1}^\infty \int_{G_k} |f_k|^p \,dA < \infty.
$$
Of course there always exist minimizing functions such
that $\| f_k \|_{p,G_k} = \| w_k \|_{E_k}$. Interpolating by a function
$f\in A^p$ then means that $f_k - f|_{G_k} \in \N(\Z_k,G_k)$ for all $k$.

\section{Properties of interpolating sequences}\label{sec:properties}

Our first observation is that interpolating sequences are zero
sequences. For a nonzero analytic function $f$ we write $\Z(f)$ for the
sequence of zeros of $f$, repeated according to multiplicity.

\begin{proposition}\label{thm:zeroset}
Given an interpolation scheme $\I$ with domains $G_k$~and clusters
$\Z_k$, if $\Z = \Union_k \Z_k$ is general interpolating sequence
for $A^p$, then there is a function $f \in A^p$ such that $\Z(f) = \Z$.
\end{proposition}

\begin{proof}
Choose a $k_0$ such that $\Z_{k_0}$ is not empty. Define a sequence in
$X^p_\I$ by selecting the function $f_{k_0}\equiv 1$ on $G_{k_0}$
and $f_k\equiv 0$ on all other $G_k$, $k \ne k_0$.

If $\Z$ is a general interpolating sequence for $A^p$ then there is
a function $g \in A^p$ which has value $1$ at all points of $\Z_{k_0}$
and $0$ at all other points (with multiplicity at least as great as that
of $\Z$). Multiply $g$ by a suitable polynomial to obtain a nontrivial
function $f$ in $A^p$ with $\Z\subset\Z(f)$. Since a subsequence of an
$A^p$ zero sequence is also an $A^p$ zero sequence (Horowitz
\cite{Hor74}) we have the result.
\end{proof}

The following is a crucial property of general interpolating
sequences. We could prove it by generalizing the proof of
Proposition~\ref{thm:discreteness}, except that we have not yet
established that the interpolation mapping from $A^p$ into $X^p_\I$ is
bounded.

\begin{theorem}[Separation]\label{thm:lowerbound}
  Given an interpolation scheme $\I$ with clusters $\Z_k$, if $\Z =
  \Union_k\Z_k$ is a general interpolating sequence for $A^p$ then
  there is a lower bound $\delta > 0$ on the pseudohyperbolic distance
  between different clusters of $\I$.
\end{theorem}

\begin{proof}
Suppose not; let $\{a_j\}$~and $\{ a_j' \}$ be subsequences of $\Z$ such
that $a_j$~and $a_j'$ lie in different clusters $\Z_j$~and $\Z_j'$ and
such that $\psi(a_j,a_j') = \eps_j \to 0$. Functions in $A^p$ satisfy
the following inequality: For a fixed positive $r<1$ there is a constant
$C_r$ such that
\begin{equation*}
   |f'(z)(1 - |z|^2)|^p \le  \frac{C_r}{|D(z,r)|} \int_{D(z,r)} |f|^p
   \,dA
\end{equation*}
Suppose $f$ has the value $\gamma_j$ at $a_j$ and $0$ at $a'_j$ for each
$j$. Then there exists a point $\zeta_j$ on the line segment connecting
$a_j$~and $a'_j$ where $|f'(\zeta_j)| \ge |\gamma_j|/|a_j - a'_j|$.
Combining this with the above inequality gives
\begin{equation*}
   |\gamma_j|^p(1 - |\zeta_j|^2)^2 \le \frac{|a_j - a'_j|^p}{(1 -
   |\zeta_j|^2)^p} \, \frac{C_r(1 - |\zeta_j|^2)^2}{|D(\zeta_j, r)|}
   \int_{D(\zeta_j, r)} |f|^p \,dA
\end{equation*}
The first fraction above is $O(\eps_j^p)$, the second is $O(1)$.
Therefore, if $f \in A^p$ and the disks $D(\zeta_j,r)$ are disjoint then
\begin{equation}\label{eq:toosmall}
  \sum_j \eps_j^{-p} |\gamma_j|^p(1 - |\zeta_j|^2)^2 \le C\| f \|_p^p <
  \infty.
\end{equation}
The disjointness of $D(\zeta_j,r)$ can easily be arranged by passing to a
subsequence if necessary. Select $\gamma_j$ so that $\sum |\gamma_j|^p(1
- |\zeta_j|^2)^2 < \infty$ but $\sum \eps_j^{-p} |\gamma_j|^p
(1-|\zeta|^2)^2 = \infty$. Define an element of $X^p_\I$ by selecting
functions that are constant on each domain $G_k$, equal to $\gamma_j$ on
the domain $G_j$ that contains $a_j$ for all $j$, and equal to 0 on
every other domain.

The sequence so defined clearly belongs to $X^p_\I$, but because of
inequality~\eqref{eq:toosmall} it cannot be interpolated by a function
in $A^p$.
\end{proof}

If we trim our interpolation scheme $\I$ by eliminating those domains
$G_k$ where $\Z_k$ is empty, it follows from the previous proposition
that the domains $G_k$ have a bounded amount of overlap. That is, the
function $\sum \chi_{G_k}$ is a bounded function. For if there is a
point $b\in\ID$ contained in $N$ of these sets, there would be at least
$N$ points from different $\Z_k$ in the disk $D(b,R)$, where $R$ is the
diameter of the scheme $\I$. By the proposition there is a $\delta > 0$
such that these $N$ points are all separated by $\delta$. These
estimates are in contradiction if $N$ can be made arbitrarily large.

For a function $f\in A^p$ we have $\sum \int_{G_k} |f|^p \,dA \le
\left\| \sum\chi_{G_k} \right\|_\infty^p \| f \|_p^p$. Thus, the
interpolation operator described above (taking $f$ to the sequence of
equivalence classes of $f|_{G_k}$) is bounded from $A^p$ into $X^p_\I$.
By the open mapping theorem, if $\Z$ is a general interpolating
sequence for $A^p$ then there is an finite \term{interpolation constant
$K$}. That is, for all $w \in X^p_\I$, there is a function $f_w$ in
$A^p$ which interpolates $w$ and satisfies the inequality $\| f_w \|_p
\le K\| w \|_{X^p_\I}$

We continue with the exposition of several properties of general
interpolating sequences. For example, they are hereditary, M\"obius
invariant and finitely extendable.

\begin{proposition}[Hereditary]\label{thm:hereditary}
  Given an interpolation scheme $\I$ with clusters $\Z_k$ and domains
  $G_k$ let $\J$ have clusters $\W_k \subset \Z_k$ and the same
  domains. If $\Z = \Union \Z_k$ is a general interpolating sequence
  for $A^p$ with interpolation constant $K$, relative to $\I$, then $\W
  = \Union \W_k$ is a general interpolating sequence for $A^p$
  relative to $\J$. Its interpolation constant is at most $K$.
\end{proposition}

\begin{proof}
Assume $\Z$ is a general interpolating sequence and $u=(u_k) \in X^p_{\J}$. Let
$E_k$ be the target spaces of $\I$ and $F_k$ denote the target spaces of
$\J$: $F_k = A^p(G_k)/\N(\W_k,G_k)$. For each $k$, select a function
$g_k$ in the equivalence class $u_k$ such that $\|g_k\|_{p,G_k}^p = \|
u_k \|_{F_k}$. The quotient mapping taking $g_k$ into $v_k \in E_k$ then
satisfies $\| v_k \|_{E_k} \le \| g_k \|_{p,G_k} = \| u_k \|_{F_k}$,
so $(v_k) \in X^p_\I$. If $f\in A^p$ interpolates $(v_k)$ that means that
$g_k - f|_{G_k} \in \N(\Z_k, G_k) \subset \N(\W_k, G_k)$ and so $f$
interpolates $(u_k)$.
\end{proof}

If we select one point from each cluster, it is easy to see that we get
a traditional interpolation problem. This is essentially how we
approached the separation between clusters: by examining a point from
each cluster.

\begin{proposition}[Invariant]\label{thm:invariant}
Let $\I$ be an interpolation scheme with clusters $\Z_k$ and domains
$G_k$. If $\Z = \Union \Z_k$ is a general interpolating sequence and
$\phi$ is a conformal map from $\ID$ onto itself, then $\phi(\Z)$ is a
general interpolating sequence relative to the scheme $\phi(\I)$
which has clusters $\phi(\Z_k)$ and domains $\phi(G_k)$. Moreover, both
sequences have the same interpolation constant.
\end{proposition}

\begin{proof}
The map $f \to f\circ\phi\cdot(\phi')^{2/p}$ is an isometry of $A^p$
onto itself, but also an isometry of $A^p(\phi(G_k))$ onto $A^p(G_k)$
for each $k$. Moreover, it takes $\N(\phi(\Z_k),\phi(G_k))$ to
$\N(\Z_k,G_k)$ so it induces an isometry of the spaces
$X^p_{\phi(\I)}$~and $X^p_{\I}$. Consequently, an interpolation problem
in one setting transforms (isometrically) into an equivalent one in the
other and the solution in $A^p$ transforms (isometrically) back into a
solution in $A^p$.
\end{proof}

The following constants associated with an interpolation scheme $\I$ are
invariant under conformal maps: the diameter, the inner radius, the
separation between clusters, and the interpolation constant. Within the
context of $A^p$ interpolation, $p$ may also be considered to be
invariant. We will refer to any constant associated to an interpolation
scheme that is invariant under conformal maps as an
\term{interpolation invariant} or simply an \term{invariant of $\I$}.

\begin{proposition}[Extendable]\label{thm:extendable}
  Let $\I$ be an interpolation scheme with clusters $\Z_k$ and domains
  $G_k$. Suppose $\Z = \Union \Z_k$ is a general interpolating sequence
  and let $z_0\in\ID$. Suppose there is an $\epsilon > 0$ such that
  $\psi(z_0,\Z_k) > \eps$ for every $k$. Define a new scheme $\J$ whose
  domains are all the domains of $\I$ plus the domain $G_0 = D(z_0,1/2)$
  and whose clusters $\W_k$ are all the $\Z_k$ plus $\W_0 = \{ z_0 \}$.
  Then $\W = \{ z_0 \} \union \Z$ is a general interpolating sequence
  relative to the scheme $\J$.

  If $K$ is the interpolation constant for $\Z$ then the constant for
  $\W$ is at most $C_pK/\epsilon$, where $C_p$ is a positive constant that
  depends only on $p$.
\end{proposition}

We will normally use this with $\epsilon = 1/4$ or $1/2$ and in that
case we see that the interpolation constant of the new sequence can be
estimated in terms of invariants of $\I$.

\begin{proof}
Because of the previous proposition, we may assume $z_0 = 0$. Note that
$E_0$, the quotient space associated with $0\in G_0$, is one
dimensional. Let $u=(u_k)$ be a sequence in $X^p_{\J}$ with $\| u
\|_{X^p_{\J}} = 1$ Let $(g_k)$ be a sequence of representative
functions satisfying $\| g_k \|_{p,G_k} = \| u_k \|_{E_k}$, $k \ge 0$.
We can identify $u_0$ with a constant function $g_0$ (this minimizes the
$L^p$ norm on a disk) and so it makes sense to consider the functions
$f_k(z) = (g_k(z) - g_0)/z$ on each $G_k$, $k \ge 1$. One easily
estimates
\begin{equation*}
    \| f_k \|_{p,G_k}^p \le \frac{C_p}{\epsilon^p} \left( \| g_k
    \|_{p,G_k}^p +  |g_0|^p|G_k| \right), \quad k\ge 1,
\end{equation*}
and so $f_k$ represents an element of $X^p_\I$. Let $f$ interpolate this
element, with $\| f \|_p^p \le K^p\sum \| f_k \|_{p,G_k}^p$. Finally,
let $h(z) = zf(z) + g_0$ and observe that it interpolates $u_k$ for $k
\ge 1$ but also satisfies $h(0) = g_0$. Since $\| g_0 \|_p^p = 4\|
g_0 \|_{p,G_0}^p \le 4\| u \|_{X^p_\J}^p$, we get the stated estimate on
the interpolation constant.
\end{proof}

We will say that the sequence $\Z$ \term{has bounded density} if there
is a positive constant $M$ and a fixed radius $0 < R < 1$ such that for
each $z\in \ID$, the disk $D(z,R)$ contains no more than $M$ terms
(counting multiplicity) of the sequence $\Z$. The number $M$ will depend
on $R$, but the actual value of $R$ will not affect whether a sequence
has bounded density.

The next theorem implies that a general interpolating sequence must
have bounded density.

\begin{theorem}[Bounded Density]\label{thm:boundeddensity}
  If $\Z$ is a general interpolating sequence for $A^p$ relative to
  an interpolation scheme $\I$ then there is a finite upper bound $B$ on
  the number of points (counting multiplicity) in each cluster $\Z_k$ of
  $\I$. The bound $B$ is obviously an invariant of $\I$.
\end{theorem}

\begin{proof}
Let $B_k$ be the cardinality of $\Z_k$ counting multiplicity and suppose
$B_k$ is not bounded.

For each $G_k$ either there exists another domain $G_k'$ with
$\psi(G_k,G_k') < 1/2$ or there is a pseudohyperbolic disk $D_k$ of
radius $1/4$ disjoint from all $G_j$ and satisfying $\psi(G_k,D_k) <
1/2$. In the first case let $a_k$ be any point in $\Z_k'$ (the cluster
contained in $G_k'$); in the second case let $a_k$ be the center of the
disk. Let $\phi_k$ be the conformal map which takes $a_k$ to the origin.
Observe that there is a radius $R' < 1$ depending only on the diameter
$R$ of $\I$, such that the number of elements $\phi_k(\Z)$ inside
$D(0,R')$ is unbounded.

For each $k$ we have an obvious scheme $\I_k$ associated with
$\phi_k(\Z) \union \{ 0 \}$: either $\phi_k(\I)$ or $\phi_k(\I)$
together with the cluster $\{ 0 \}$ and the domain $D(0,1/4)$. We can
define an element of $X^p_{\I_k}$ by choosing functions that are $1$ in
the domain containing the origin and zero in all other domains. The norm
of this element is at most $1$ and, by propositions
\ref{thm:hereditary}~through \ref{thm:extendable}, there exist a bound
depending only on $p$ and $K$ (independent of $k$) on the $A^p$ norm of
an interpolating function for this element. Thus we obtain a bounded
sequence $f_k$ in $A^p$ such that $f_k(0) = 1$ and the number of zeros
in the relatively compact disk $D(0,R')$ is unbounded. The former
condition prevents any subsequence of $f_k$ from converging to $0$
uniformly on compacta, while the latter condition implies that there
does exist such a subsequence. This contradiction proves $B_k$ is
bounded.
\end{proof}

In fact, given an upper bound on the $A^p$ norm of a function $f_k$ and
a lower bound on its value at 0 (namely $1$), Jensen's formula provides
an upper bound on the number of zeros of the function in $D(0,R')$.
Therefore the bound $B = \sup_k B_k$ actually depends only on $p$,
$K$ and $R$.

To see that a general interpolation scheme has bounded density,
let $\delta>0$ be less than half the distance between distinct clusters.
If $D(z,R)$ is some pseudohyperbolic disk, the disjointness of
$(\Z_k)_\delta$ provides an upper bound (depending only on $\delta$ and
$R$) on the number of clusters that meet $D(z,R)$. Since there is an
upper bound on the number of points in each cluster, we get an upper
bound on the number of points of $\Z$ in $D(z,R)$. Clearly this upper
bound (for fixed $R$) is also an interpolation invariant.

\section{Admissible interpolation schemes}\label{sec:admissible}

Let $\I$ be an interpolation scheme with domains $G_k$ and clusters
$\Z_k$. The following lists the two properties of $\I$ that we required
in the definition of an interpolation scheme, together with the
properties obtained in Theorems \ref{thm:lowerbound}~and
\ref{thm:boundeddensity} for general interpolation sequences.
\begin{itemize}
\item[(P1)] There is an $R < 1$ such that the pseudohyperbolic diameter
     of each $G_k$ is at most $R$.
\item[(P2)] There is an $\epsilon > 0$ such that $(\Z_k)_\epsilon \subset
    G_k$.
\item[(P3)] There is a $\delta > 0$ such that for all $j\ne k$ the
    pseudohyperbolic distance from $\Z_j$ to $\Z_k$ is at least
    $\delta$.
\item[(P4)] There is an upper bound $B$ on the number of points (counting
    multiplicity) in each cluster $\Z_k$
\end{itemize}

We will say that an interpolation scheme is \term{admissible} if it
satisfies, in addition to (P1)~and (P2), the conditions (P3)~and (P4).
We will call a sequence $\Z$ admissible if there is an admissible scheme
with $\Z = \Union_k \Z_k$. We will refer to $\delta$ as a
\term{separation constant of $\I$} and $B$ as a \term{cluster bound}

Verifying that a sequence is admissible would seem to require examining
a great many possible choices of interpolation schemes. However, the
following shows that the class of admissible sequences coincide with an
already very familiar class of sequences.

\begin{theorem}
A sequences is admissible if and only if it has bounded density.
\end{theorem}

\begin{proof}
We saw earlier that properties (P3)~and (P4) imply bounded density. For
the converse, assume $\Z$ has bounded density.

Pick any $0< r_0 < 1$ and let $B$ be a bound on the number of points in
$D(z,r_0)$, $z\in \ID$. We will show that there exists $\epsilon > 0$ such
that we can take $G_k$ to be the connected components of $(\Z)_\epsilon$
and $\Z_k = \Z\intersect G_k$. Property (P2) is obvious, and so is (P3)
with $\delta = 2\eps$.

If we are able to show (P1) for these connected components, then (P4)
will trivially follow since $\Z$ has bounded density. So we are reduced
to showing that for sufficiently small $\epsilon > 0$ a connected
component of $(\Z)_\epsilon$ has diameter at most $R < 1$. We do this by
showing that each $G_k$ is contained in some disk $D(z, r_0)$, and then
$R$ will be $2r_0/(1+r_0^2)$, the diameter of $D(z,r_0)$.

Let $G_k$ be some component of $(\Z)_\epsilon$. Clearly $G_k =
(\Z_k)_\epsilon$ for some subset $\Z_k$ of $\Z$. Select a point
$z_0\in\Z_k$ and consider circles $\Gamma_j$, each being the boundary of
$D(z_0, \eps_j)$ where $\epsilon_1 < \epsilon_2 <\dots<\epsilon_{B+1}$.
Choose these radii so that $\eps_1 = \eps$ and $\eps_{j+1} =
(\eps_j+\eps')/(1 + \eps_j\eps')$ where $\eps'$ is the pseudohyperbolic
diameter of $D(z,\eps)$. Note that a disk $D(z,\eps)$ cannot intersect
more than one of these circles.

Now select $\eps$ so small that $\eps_{B+1} \le r_0$ and assume, for the
purpose of contradiction, that $G_k$ is not contained in $D(z_0,r_0)$. In
that case, $G_k$ must intersect all of the $\Gamma_j$ and so there must
exist points $z_j$, $1 \le j \le B$ with the disk $D(z_j,\eps)$
intersecting $\Gamma_j$. This forces $B+1$ points ($z_0$ through $z_B$)
of $\Z$ to be inside of $D(z_0,r_0)$, where there are assumed to be only
$B$ points. This contradiction implies that $G_k$ is contained in
$D(z_0,r_0)$.
\end{proof}

Given a sequence $\Z$ with bounded density, there are essentially two
extreme choices of domains $G_k$ satisfying properties (P1)~through
(P4). A ``minimal one'' is to let $G_k$ be the connected components of
$(\Z)_\epsilon$ for sufficiently small $\eps$; a ``maximal one'' is to
let $G_k = D(z_k, r)$ where $D(z_k, r)$ is a disk containing the $k$th
connected component of $(\Z)_\epsilon$. These differ in the norm imposed
on the finite dimensional spaces $E_k$. One suspects that these norms
might impose equivalent norms on $X^p_\I$ (but we do not have a proof).
We will however show that whether or not $\Z$ is a general
interpolating sequence for $A^p$ is independent of the choice. If $\Z$
is a general interpolating sequences then the norms $\| \cdot
\|_{X^p_\I}$) are equivalent for all admissible $\I$, for in that case they
are all equivalent to the quotient space norm in $A^p/\N(\Z)$ where
$\N(\Z)$ is the subspace of functions vanishing on $\Z$.

\section{General interpolating sequences, density and the
\texorpdfstring{$\dbar$}{dbar}-problem}\label{sec:density}

The following is one of the main steps in nearly all characterizations
of interpolating sequences: some key property is preserved under small
perturbations of the sequence. In the present context we have the mild
additional complication that a perturbation of a sequence also perturbs
the interpolation scheme chosen for it.

\begin{proposition}[Stability]\label{thm:stability}
  Let $\I$ be an admissible interpolation scheme with domains $G_k$,
  clusters $\Z_k$ and diameter $R$. Assume $\Z = \Union_k \Z_k$ is a
  general interpolation sequence for $A^p$ with interpolation
  constant $K$. For each $k$ let $\phi_k$ be defined by $\phi(z) = r_k
  z$ and let $\J$ be the interpolation scheme with domains $\phi_k(G_k)$
  and clusters $\W_k = \phi_k(\Z_k)$. Let $\W =
  \Union_k \W_k$.

  There exists an $\eta > 0$ depending only on interpolation invariants
  such that if $\psi(\phi_k(z),z) < \eta$ for all $z\in G_k$ and for all
  $k$, then $\W$ is a general interpolation sequence for $A^p$
  relative to $\J$. Its interpolation constant can be estimated in terms
  of $\eta$ and interpolation invariants of $\I$.
\end{proposition}

\begin{proof}
It is clear that if we choose $\eta$ small enough, then $\J$ is
admissible since all points of $G_k$ are moved a pseudohyperbolic
distance at most $\eta$. We need only take $\eta < \delta/4$, where
$\delta$ is a separation constant of $\I$, and then $\delta/2$ is a
separation constant for $\J$.

As usual, let $E_k = A^p(G_k))/\N(\Z_k,G_k)$. Let $D_k = \phi_k(G_k)$
and let $F_k = A^p(D_k))/\N(\W_k,D_k)$. Recall $\rho(G_k) = \sup
\{ 1 - |z| \st z\in G_k \}$. The hypotheses imply that $1 - r_k$ are
small multiples of $\rho(G_k)$.

Let $v = (v_k)$ be any sequence in $X^p_\J$. We can transfer this, via
$\phi_k$ to a sequence $u = (u_k)$ in $X^p_\I$: let $u_k$ be the element
of $E_k$ represented by $f_k\circ\phi_k$ where $f_k$ represents $v_k \in
F_k$. It will be convenient to omit the factor $(\phi_k')^{2/p} =
r_k^{2/p}$, so this is not an isometry. However, it is still an
isomorphism from $X^p_\J$ onto $X^p_\I$ and there is a constant $C$ with
$\| u \|_{X^p_\I} \le C\| v \|_{X^p_\J}$.

For $f \in A^p$, let $\Phi(f)$ denote the element of $X^p_\J$
represented by the sequence $(f|_{D_k})$. It suffices to obtain a
function $f_v \in A^p$ such that
\begin{align}
  \| v - \Phi(f_v)\|_{X^p_\J} \le \frac{\| v \|_{X^p_\J}}{2} \label{eq:closeness}\\
\intertext{and}
  \| f_v \|_p \le M\| v \|_{X^p_\J}\label{eq:norm}
\end{align}
for some finite $M$. For if this can be done, iterating the process as
in \cite{Lue85a} produces a sequence $v^{(j)}$ in $X^p_\J$ and a
corresponding sequence $f_{v^{(k)}}$ in $A^p$ satisfying $\| f_{v^{(j)}}
\|_p < M\| v \|_{X^p_\J}/2^j$ and $\| v - \sum_{j=1}^n \Phi(f_{v^{(j)}})
\|_{X^p_\J} < \| v \|_{X^p_\J}/2^n$. Thus, with $f = \sum_j
f_{v^{(j)}}$, we have $\Phi(f)=v$.

We take $f = f_v$ to be the function that interpolates $(u_k)$ with $\|
f \|_p \le K\| u \| \le CK\| v \|_{X^p_\J}$. Then $(g_k) =
(f\circ\phi_k^{-1})$ represents $v_k$ and it suffices to estimate the
difference $\| g_k - f|_{D_k} \|_{p,D_k}$:
\begin{equation}\label{eq:A1}
    \| g_k - f|_{D_k} \|^p_{p,D_k}
        = \int_{D_k} |f(z/r_k) - f(z)|^p \,dA(z).
\end{equation}
If $\eta < 1/4$, we have (see, for example, \cite{Lue85a})
\begin{equation*}
    |f(z/r_k) - f(z)|^p \le \frac{C_p\eta}{|D(z, 1/2)|}
    \int_{D(z,1/2)} |f|^p \,dA,
\end{equation*}
Integrating this gives
\begin{equation}\label{eq:A2}
    \int_{D_k} |f(z/r_k) - f(z)|^p \,dA(z) \le C_p\eta \int_{D'_k}
     |f|^p \,dA,
\end{equation}
where $D'_k = (D_k)_{1/2}$ is the domain $D_k$ expanded by
pseudohyperbolic distance $1/2$. There is a number $B'$ (depending on
the diameter and separation constants of $\I$) such that at most $B'$ of
the domains $D'_k$ overlap at any point. Combining
equation~\eqref{eq:A1} and inequality~\eqref{eq:A2} and summing, we have
\begin{align*}
    \| v - \Phi(f) \|^p_{X^p_\W}
        &\le \sum_k \| g_k - f|_{D_k} \|^p_{p,D_k} \\
        &\le C_p\eta\sum_k \int_{D'_k} |f|^p \,dA  \\
        &\le C_p\eta B'\| f \|_p \le C_p\eta B'CK\| v \|_{X^p_\W}.
\end{align*}
If we take $\eta < 1/(2C_pB'CK)$ we have \eqref{eq:closeness} and
\eqref{eq:norm} with $M = CK$, as required.
\end{proof}

The definition of \term{upper uniform density} can be found in
\cite{Sei93b}. Here is an equivalent definition. For a sequence $\Z$ in
the unit disk, let
\begin{equation*}
  D(\Z, r) = \left( \frac{1}{2}\sum_{|z_k| < r} (1  - |z_k|^2)
  \right) \biggm/ \left( \log \frac{1}{1 - r^2} \right).
\end{equation*}
and let
\begin{equation*}
    D^+(\Z) = \limsup_{r\to 1} \left[ \sup_\phi D(\phi(\Z),r) \right]
\end{equation*}
where the supremum is taken over all conformal self-maps of $\ID$. The
quantity $D^+(\Z)$ is called the \term{upper uniform density} of $\Z$.

Because of its connection with the $\dbar$-problem posed in the
introduction, we will in fact use an equivalent version of $D^+(\Z)$.
Recall that
\begin{equation*}
    k_\Z(\zeta) = \frac{|\zeta|^2}{2} \sum_{k = 1}^\infty \frac{(1 -
    |z_k|^2)^2}{|1 - \bar z_k\zeta|^2}.
\end{equation*}
It is easy to compute the average of $k_\Z(\zeta)$ over the circle
$|\zeta| = r$. It is
\begin{equation*}
    \hat k_\Z(r) = \frac{1}{2\pi}\int_0^{2\pi} k_\Z(re^{it}) \,dt =
    \frac{r^2}{2} \sum_{k=1}^\infty \frac{(1 - |z_k|^2)^2}{1 -
    |z_k|^2r^2}
\end{equation*}

The following was shown in \cite{Lue04a}. There it was assumed that $\Z$
was uniformly discrete, but the proof only made use the fact that it had
bounded density.

\begin{Athm}
  For a sequence $\Z$ without limit points in $\ID$
  \begin{equation*}
    D^+(\Z) = S^+(\Z) \defeq \limsup_{r \to 1} \left[ \sup_\phi \left( \hat
        k_{\phi(\Z)}(r) \right)\biggm/ \left(\log\frac{1}{1 -
        r^2}\right) \right].
  \end{equation*}
\end{Athm}

Note that the finiteness of either $D^+(\Z)$ or $S^+(\Z)$ is equivalent
to $\Z$ having bounded density, so that does not need to be part of the
hypotheses.

In order to deal with the case $p < 1$ we define a space $L^p_{q,\Z}$
that will (as in \cite{Lue04a}) serve as a replacement for $L^p_\Z$.
Choose some radius $r < 1$, then a measurable function $f$ belongs to
$L^p_q$ if and only if the function $m_q(f,z)$ belongs to $L^p$, where
\begin{equation*}
    m_q(f,z)^q = \frac{1}{|D(z,r)|} \int_{D(z,r)} |f|^q \,dA.
\end{equation*}
The radius $r$ should be fixed, but the actual value chosen is not
important. Define $L^p_{q,\Z}$ to be the set of all $f$ such that
$fe^{k_\Z}$ belongs to $L^p_q$. We will always assume $q \ge 1$. Note
that $L^p_p = L^p$ and $L^p_{p,\Z} = L^p_\Z$. Also, the set of analytic
functions in $L^p_q$ (resp., $L^p_{q,\Z}$) is independent of $q$, being
$A^p$ (resp., $A^p_\Z$). We will allow $q = \infty$ with $m_\infty(f,z) =
\sup_{w\in D(z,r)} |f(w)|$.

Finally, we have the following equivalent conditions for general
interpolation.

\begin{theorem}\label{thm:realmain}
  Let $\Z$ be a sequence in $\ID$ without limit points. The following
  are equivalent:
  \begin{enumerate}
    \item $\Z$ has bounded density and is a general interpolating
          sequence for $A^p$ relative to any admissible interpolation
          scheme.\label{GIS1}
    \item $\Z$ is a general interpolation sequence for $A^p$
          relative to some interpolation scheme.\label{GIS2}
    \item $D^+(\Z) < 1/p$ or equivalently $S^+(\Z) < 1/p$.\label{UUD}
    \item $\Z$ has bounded density and the $\dbar$-problem has solutions
        with bounds in $L^p_{q,\Z}$ for all $q \ge 1$.\label{DBAR1}
    \item $\Z$ has bounded density and the $\dbar$-problem has solutions
        with bounds in $L^p_{q,\Z}$ for some $q \ge 1$.\label{DBAR2}
  \end{enumerate}
\end{theorem}

It is trivial that \ref{GIS1} implies \ref{GIS2}  and \ref{DBAR1}
implies \ref{DBAR2}. In this section we will prove that \ref{GIS2}
implies \ref{UUD}. In the next section we will show \ref{UUD} implies
\ref{DBAR1} and \ref{DBAR2} implies \ref{GIS1}.

\begin{proof}[Proof of \textup{\ref{GIS2} $\Rightarrow$ \ref{UUD}}]
Assume \ref{GIS2} and let $\I$ be the interpolation scheme with domains
$G_k$, clusters $\Z_k$ and diameter $R$. We have seen that $\I$ is
necessarily admissible and $\Z$ necessarily has bounded density. Let $K$
be the interpolation constant for $\Z$. The proof begins almost exactly
as in \cite{Lue04a}. We wish to show the existence of a function $g$ in
the unit ball of $A^{p/\beta}$, for some $\beta < 1$, that vanishes on
``most'' of $\Z$ and satisfies $|g(0)| > \delta$, where $\delta>0$
is an interpolation invariant.

Start by removing from $\Z$ those clusters $\Z_k$ for which $G_k$
intersects $D(0,1/2)$. The number of points removed is bounded by an
amount which is an interpolation invariant. Then append the single point
$\{ 0 \}$ to the result. Call this new sequence $\W$. By Propositions
\ref{thm:hereditary}~and \ref{thm:extendable}, $\W$ is a general
interpolating sequence relative to a related interpolation scheme. It
has an interpolation constant which may be estimated solely in terms of
interpolation invariants of $\Z$.

Now perturb the sequence $\W$ as in Proposition~\ref{thm:stability},
with the affine maps $\phi_k(z) = r_k z$, where each $0 < r_k < 1$ and
$1 - r_k$ is a fixed small multiple $\eta$ of $\rho(G_k)$. If $\eta$ is
sufficiantly small, depending only on interpolation invariants, the
result is a general interpolating sequence. Call the new sequence
$\W'$. If $a$ is a point of $\W$ the perturbation $a'$ in $\W'$ will
satisfy $(1 - |a|^2) < \beta (1 - |a'|^2)$ for some $\beta < 1$
depending only on $\eta$ and $R$. Since this new sequence is a zero set
for $A^p$ then, according to \cite{Lue00b} and \cite{Lue04a}, perturbing
it back to the original $\W$ produces a sequence that is a zero sequence
for $A^{p/\beta}$. (Strictly speaking, Theorem~4 of \cite{Lue00b}
assumed $1 - |a|^2 = \beta (1 - |a'|^2)$ for all $a$, but only the
inequality is actually needed in the proof.)

Moreover, since $\W'$ is interpolating, there is a function $f$ in $A^p$
with norm $1$ whose zero sequence is $\W'\setminus\{ 0 \}$ and satisfies
$|f(0)| > \delta$ where $\delta > 0$ depends only on interpolation
invariants for $\I$. As in \cite{Lue04a}, there can be constructed a
function $g$ in $A^{p/\beta}$ whose zero sequence is $\W\setminus\{ 0
\}$ and which satisfies $\| g \|_{p/\beta} \le C\| f \|_{p}$ and $|g(0)|
> |f(0)|/C$, where $C$ depends only on interpolation invariants.

Also as in \cite{Lue04a}, we can produce a function $h \in
A^{p/\beta}_\Z$ ($\Z$ the original sequence, which differs from $\W$ by
a number of elements that depends only on invariants of $\I$) where $h$
vanishes nowhere and satisfies $\| h \|_{p/\beta,\Z} = 1$ and $|h(0)| >
\delta$ with some (new) value of $\delta > 0$ depending only on
invariants of $\I$. Let $h^*$ be that nonvanishing function in the unit
ball of $A^{p/\beta}_\Z$ which maximizes the absolute value at the
origin. As shown in \cite{Lue04a}, it will satisfy
\begin{equation}\label{eq:growth}
    |h^*(\zeta)|^{p/\beta} e^{(p/\beta)k_\Z(\zeta)} (1 - |\zeta|^2) \le C
\end{equation}
with an absolute constant $C$.

Now consider
\begin{multline*}
  \frac{1}{2\pi}\int_0^{2\pi} \frac{p}{\beta} \log |h^*(re^{it})| +
        \frac{p}{\beta}k_\Z(re^{it}) + \log (1 - |r|^2) \,dt \\
  \begin{aligned}
    &\le \log \left( \frac{1}{2\pi} \int_0^{2\pi} |h^*(re^{it})|^{p/\beta}
        e^{(p/\beta)k_\Z(re^{it})}(1 - |r|^2) \,dt \right) \\
    &\le \log C,
  \end{aligned}
\end{multline*}
valid for all $0 < r < 1$. Divide the above by $\log[1/(1-r^2)]$ and
$p/\beta$ to obtain:
\begin{equation}\label{eq:Zestimate}
  \left( \hat
        k_{\Z}(r) \right)\biggm/ \left(\log\frac{1}{1 - r^2}\right)
        \le \frac{\beta}{p}\left(1 + \frac{\log C -
        \log|h^*(0)|}{\log\frac{1}{1 - r^2}}\right).
\end{equation}
Fix a $\kappa> 0$ so that $\beta(1 + \kappa) < 1$ and choose
a sufficiently large $r_* < 1$ so that large parenthesis on the right side
of the above inequality is less than $1 + \kappa$. Since $-\log|h^*(0)|
\le \log(1/\delta)$, the choice of $r_*$ depends only on interpolation
invariants. Applying this to \eqref{eq:Zestimate} with $\phi(\Z)$ in
place of $\Z$ gives
\begin{equation*}
  \left( \hat
        k_{\phi(\Z)}(r) \right)\biggm/ \left(\log\frac{1}{1 - r^2}\right)
        \le \frac{\beta}{p}(1 + \kappa), \quad \text{all }r > r_*,
\end{equation*}
for all conformal maps $\phi$ of the unit disk. Taking the supremum over
all such $\phi$ gives us $S^+(\Z) < 1/p$, as required for
condition~\ref{UUD}.
\end{proof}

\section{The \texorpdfstring{$\dbar$}{dbar}-problem and general interpolating
sequences}\label{sec:interpolating}

The existance of $h^*$ satisfying inequality~\eqref{eq:growth} is
sufficient (when conformal invariance is invoked) to prove
condition~\ref{DBAR1} of Theorem~\ref{thm:realmain}. So now we want to
prove that condition~\ref{UUD} implies the existence of such a function.

We \emph{could} argue as in the last section of \cite{Lue04a} that this
density condition implies the hypotheses of a theorem of
J.~Ortega-Cerd\`a (\cite{Ort02}) which implies the weighted
$\dbar$-estimates. That, however, turns out to be a much stronger result
than we need. Moreover, the paper \cite{Ort02} does not explicitly
include the case $p < 1$. Therefore, we will prove directly that the
density condition~\ref{UUD} implies the existence of an analytic
function with the growth property of \eqref{eq:growth}.

Define the invariant Laplacian $\invL f(z) = (1 - |z|^2)^2\laplace
f(z)$.

\begin{lemma}\label{thm:harmonicbound}
  If $\tau$ is a real valued $C^2$ function in $\ID$ such that $\invL
  \tau$ is negative and bounded, then there exists a harmonic function
  $v$ such that $\tau(z) - v(z) \le 0$ and $\tau(0) - v(0) \ge -C\| \invL \tau
  \|_\infty$, where $C>0$.
\end{lemma}

\begin{proof}
It suffices to prove that the  equation $\laplace u = \laplace\tau$ has
a solution $u$ that is bounded above, with an estimate on the value at
$0$. Without any loss of generality we can take $\tau(0) = 0$. We simply
write down a formula for $u$ and verify its properties:
\begin{equation*}
  u(z) = \frac{2}{\pi} \int_{\ID} \invL \tau(w) \left[ \log\left|\bar w
        \frac{w - z}{1 - \bar w z}  \right| + \re \frac{1 - |w|^2}{1 -
        \bar w z} \right] \,d\lambda(w),
\end{equation*}
where $\lambda$ denotes the invariant area measure on $\ID$.
Assuming the integral makes sense, this surely satisfies the equation
$\laplace u = \laplace \tau$. We can rewrite the expression in brackets as
\begin{multline*}
     \log\left|\bar w \frac{w - z}{1 - \bar w z}  \right| + \re \frac{1
        - |w|^2}{1 - \bar w z} ={}\\ \left[ \log |w| + \frac{1 -
        |w|^2}{2} \right] + \left[ \log \left| \frac{w - z}{1 - \bar wz}
        \right| + \frac{1}{2}\left( 1 - \left| \frac{w - z}{1 - \bar wz}
        \right|^2 \right) \right] + \frac {|z|^2 (1 - |w|^2)^2} {2
        |1-\bar w z|^2}.
\end{multline*}
Let us examine these three expressions in reverse order. The last one
times $\invL \tau(w)$ is clearly integrable with respect to
$d\lambda(w)$ and the resulting integral is negative. The first two
bracketed expressions times $\invL \tau(w)$ are positive, but I claim
the integrals are bounded. The second is integrable $d\lambda(w)$ in
each disk $D(z,R)$ (with the integral independent of $z$), and outside such
disk it decays like
\begin{equation*}
    \left( 1 - \left| \frac{w-z}{1 - \bar wz} \right|^2 \right)^2 =
    \frac{(1 - |z|^2)^2(1 - |w|^2)^2}{|1 - \bar wz|^4}
\end{equation*}
whose integral with respect to $d\lambda(w)$ is a finite constant.
Finally, the first bracketed expression is just the second one at $z =
0$. Putting these together, we see that $u(z)$ is the sum of a negative
function and a bounded positive one, and the bounded one has absolute
value at most $C\| \invL\tau \|_\infty$ where $C$ is an absolute
constant. Finally, the value of $u(0)$ is similarly bounded by a
multiple of $\| \invL\tau \|_\infty$.

Thus, if we define $v(z) = \tau(z) - u(z) + C\| \invL\tau \|_\infty$ for
a suitable absolute constant $C$, then $v$ satisfies the stated
requirements.
\end{proof}

We apply this to a function which is almost
\begin{equation}\label{eq:definetau}
    \tau(\zeta) = \log \frac{1}{1 - |\zeta|^2} - \frac{p}{\beta}
        k_\Z(\zeta)
\end{equation}
This does not quite satisfy $\invL \tau < 0$, but the argument in
\cite{Lue04a} shows that a certain smoothing of it does satisfy this.
We continue the proof of the main result:

\begin{proof}[Proof of \textup{\ref{UUD} $\Rightarrow$ \ref{DBAR1}
$\Rightarrow$ \ref{DBAR2} $\Rightarrow$ \ref{GIS1}}]
Let $\tau(z)$ be as in \eqref{eq:definetau}. If condition \ref{UUD} is
satisfied we have, according to \cite{Lue04a}, the following inequality
for some radius $r_* > 0$:
\begin{equation*}
    \frac{1}{\pi \log[1/(1 - r_*^2)]} \int_{D(z,r_*)} \invL \tau(w) \log
        \frac {r_*^2} {\left| \frac{w - z}{1 - \bar w z} \right|^2}
        d\lambda(w) \le 0
\end{equation*}
Since the invariant Laplacian commutes with invariant convolution by
radially symmetric functions, this says that the function $\tau^*$
defined below has negative Laplacian.
\begin{equation*}
    \tau^*(z) = \frac{1}{\pi \log[1/(1 - r_*^2)]} \int_{D(z,r_*)}
        \tau(w) \log \frac {r_*^2} {\left| \frac{w - z}{1 - \bar w z}
        \right|^2} dA\lambda(w)
\end{equation*}
It is straightforward to see that $\invL \tau$ is bounded and that
$\tau^* - \tau$ is bounded. The bounds depend only on $r^*$ and the
density of the sequence $\Z$. Condition \ref{UUD} implies that $\Z$ has
bounded density. Moreover, the density of $\Z$ and the value of $r^*$ is
clearly the same for all $\phi(\Z)$ where $\phi$ is any conformal
self-map of $\ID$. Thus $\tau^*$ satisfies the hypothesis of
Lemma~\ref{thm:harmonicbound} and therefore the conclusion. Clearly
$\tau$ also satisfies the conclusion since it differs from $\tau^*$ by a
bounded function.

Given the harmonic upper bound $v(z)$ to $\tau(z)$ we obtain a non vanishing
analytic function $h$ such that $\log|h(z)| = -v(z)$. This gives
\begin{equation*}
    \left| h(z) e^{k_\Z}\right|^{p/\beta}(1 - |z|^2) \le 1
        \quad \text{and} \quad
    |h(0)| \ge \delta > 0,
\end{equation*}
where $\delta = e^{-v(0)}$ depends only on the data provided by
condition~\ref{UUD}. Moreover, the same is true if $\Z$ is replaced with
$\phi(\Z)$ for any conformal $\phi$, with the same value of $\delta$.

Now the same argument in \cite{Lue04a} applies: we can obtain a family
$g_a$, $a \in \ID$ of analytic functions satisfying the following for
some fixed positive numbers $\eps$, $\delta$, $\eta$ and $C$ (independent of
$a$):
\begin{gather}
    g_a(z)e^{k_\Z(z)} > \delta > 0, \quad z\in D(a,\eta)\\
\intertext{and}
    |g_a e^{k_\Z}|^{p} \left( 1 - \left| \frac{a-z}{1 -\bar az}
        \right|^2 \right)^{1 - \eps} \le C.
\end{gather}
Then this family allows one to construct a solution operator with the
required $L^p_{q,\Z}$ bounds. This yields condition~\ref{DBAR1}.

Clearly condition~\ref{DBAR1} implies condition~\ref{DBAR2}.

Finally, assume condition~\ref{DBAR2}. Let $\I$ be any admissible
interpolation scheme and let $w=(w_k)$ be a sequence in $X^p_\I$ with
representative functions $g_k$ satisfying $\| g_k \|_{p,G_k} = \| w_k
\|_{E_k}$. Let $\eps>0$ be chosen so that $(\Z_k)_{2\eps}
\subset G_k$. Without loss of generality, $\eps$ is less than the
separation of $\I$ so that $(\Z_k)_{\eps}$ are disjoint. It is easy to
verify the following for any $q \ge 1$:
\begin{equation*}
    \left(\frac{1}{|(\Z_k)_{\eps}|}\int_{(\Z_k)_{\eps}} |g_k|^q
    \right)^{1/q} \le \sup_{z\in(\Z_k)_{\eps}} |g_k(z)| \le
    \left(\frac{C}{|G_k|}\int_{G_k} |g_k|^p  \right)^{1/p},
\end{equation*}
where $C$ depends only on $\eps$ and the diameter of the
interpolation scheme $\I$.
For each $k$ let $\gamma_k$ denote a positive $C^1$ function supported in
$(\Z_k)_{\eps}$ which equals $1$ in $(\Z_k)_{\eps/2}$ and satisfies
$\grad \gamma_k(z) \le C/(1 - |z|)$. Then the function
$g = \sum_j \gamma_j g_j$ agrees with $g_k$ in a neighborhood of $\Z_k$. We
want to correct $g$ by putting $f = g - u\Psi_Z$ where $\Psi_\Z$ is the
function defined in \cite{Lue04a} which has zero sequence $\Z$, choosing
$u$ so that $f$ is in $A^p$. This requires $\dbar u = \dbar g/\Psi_Z$.

The argument of \cite{Lue04a} shows that $(1 - |z|^2)\dbar g/\Psi_\Z$
belongs to $L^p_{q,\Z}$. By condition~\ref{DBAR2} there is a solution
$u$ in $L^p_{q,\Z}$. It follows that $u\Psi_\Z$ is in $L^p_q$ and
therefore $f = g-u\Psi_\Z \in L^p_q$. Being analytic, it must belong to
$A^p$. Since both $g$ and $u$ are analytic in a neighborhood of each
$\Z_k$, the formula for $f$ shows that it agrees with $g$ to order at
least that determined by $\Z$ and so $f|_{G_k}$ is equivalent to $g_k$
for all $k$. That is, $f$ interpolates the sequence $(w_k)$, as
required.
\end{proof}

\section{Remarks and examples}\label{sec:examples}

\begin{remark}
Note that the proofs and results of the previous sections are
practically not changed at all for the weighted Bergman space
$A^{p,\alpha}$, the measure being $dA_\alpha(z) = (1 - |z|^2)^\alpha
dA(z)$ with $\alpha > -1$. The appropriate density inequality is then
$D^+(\Z) < (\alpha + 1)/p$.
\end{remark}

\begin{remark}
There are other norms than the one described on $X^p_\I$ that
can be used without changing the results. For example, let us rewrite
our definition of the norm in $X^p_\I$ in terms of
\emph{averages} instead of integrals. That is, define
\begin{equation*}
  \avg_q(f,k) =
  \begin{cases}
    \left( \frac{1}{|G_k|} \int_{G_k} |f|^q \,dA \right)^{1/q} &
      0<q<\infty\\
    \sup_{z\in G_k} |f(z)| & q = \infty
  \end{cases}
\end{equation*}
and for $w_k\in E_k$ let
\begin{equation*}
    \nu_q(w_k)  = \inf \{\avg_q(f,k) \st f \text{ is a representative of }
    w_k \}.
\end{equation*}
Then for $w = (w_k) \in X^p_\I$ we have
\begin{equation*}
    \| w \|_{X^p_\I}^p = \sum \nu_p(w_k)^p|G_k|.
\end{equation*}
However, if we examine the proof of the step \ref{DBAR2} $\Rightarrow$
\ref{GIS1} of Theorem~\ref{thm:realmain}, we see that at least for that
step we could have used the norm
\begin{equation*}
    \| w \| = \sum \nu_q(w_k)^p|G_k|
\end{equation*}
with any value of $q$. It is not hard to see that the necessary
analogues of all the results in sections~\ref{sec:properties} and
\ref{sec:density} go through for this altered norm. This is to be
expected for the following reason: one easily sees that $\avg_q(f,k)\le
C_{p,q}\avg_p(f,k)$ (whatever the values of $p$ and $q$) if one uses a
slightly smaller domain on the left or a slightly larger one on the
right. Thus, since the equivalent condition \ref{UUD} of
Theorem~\ref{thm:realmain} is independent of the actual domains chosen,
one should get the same result with the above norm for any choice of
$q$. The choices $q=2$ or $q=\infty$ can make calculations of the
$\inf_f \avg_q (f,k)$ particularly simple in some special cases.
\end{remark}

\begin{remark}
The main result also has a version for $p = \infty$; that is, for the
spaces $A^{-\alpha}$, defined for $\alpha > 0$ by
\begin{equation*}
    A^{-\alpha} = \left\{ f : f \text{ is analytic and } \| f
    \|_{\infty,\alpha} = \sup_{z\in\ID} |f(z)|(1 - |z|^2)^\alpha <
    \infty \right\}.
\end{equation*}
One can define a general interpolation problem in an analogous way,
substituting weighted sup norms for $L^p$ norms and suprema for
summations. Then appropriate analogues of \ref{thm:zeroset} through
\ref{thm:boundeddensity} and \ref{thm:stability} have analogous proofs
(for Proposition~\ref{thm:zeroset} we invoke \cite{Lue96} at the end
instead of \cite{Hor74}). The proof of the appropriate analogue of
Theorem~\ref{thm:realmain} is then nearly the same. If a zero set for
$A^{-\alpha}$ with bounded density is perturbed toward the boundary by a
factor $\beta < 1$, the resulting sequence is a zero sequences for
$A^{-\beta\alpha}$. The construction of functions satisfying the growth
in \eqref{eq:growth} is the same and the solution operator constructed
in \cite{Lue04a} is easily seen to be bounded in the appropriate sup
norm.

For the record, we include below the statement of the main theorem for
$A^{-\alpha}$. We let $L^{-\alpha}$ denote the set of measurable
functions $f$ such that $|f(z)|(1 - |z|^2)^\alpha$ is essentially
bounded. Let $L^{-\alpha}_q$ denote those functions where $m_q(f) \in
L^{-\alpha}$ and $L^{-\alpha}_{q,\Z}$ are those  with $fe^{k_\Z}\in
L^{-\alpha}_q$.

Recall that K.~Seip's criterion \cite{Sei93b} for simple interpolation
in $A^{-\alpha}$ is that the sequence be uniformly discrete and $D^+(\Z)
< \alpha$.

\begin{theorem}
  Let $\Z$ be a sequence in $\ID$ without limit points. The following
  are equivalent:
  \begin{enumerate}
    \item $\Z$ has bounded density and is a general interpolating
          sequence for $A^{-\alpha}$ relative to any admissible interpolation
          scheme.
    \item $\Z$ is a general interpolation sequence for $A^{-\alpha}$
          relative to some interpolation scheme.
    \item $D^+(\Z) < \alpha$ or equivalently $S^+(\Z) < \alpha$.
    \item $\Z$ has bounded density and the $\dbar$-problem has solutions
        with bounds in $L^{-\alpha}_{q,\Z}$ for all  $1\le q\le\infty$.
    \item $\Z$ has bounded density and the $\dbar$-problem has solutions
        with bounds in $L^{-\alpha}_{q,\Z}$ for some $1\le q\le\infty$.
  \end{enumerate}
\end{theorem}
\end{remark}

\begin{remark}
One can approach Hardy spaces in a similar way. For those, it seems
natural to choose the domains $G_k$ to have a bounded \emph{perimeter}
in the hyperbolic metric (rather than just bounded radius), and define
the norm on the space $E_k$ to be the infimum of $L^p$ norms on the
boundary. If one does that one obtains a notion of general
interpolating sequence which turns out to be equivalent to the measure
$\sum (1 - |z_j|^2)\delta_{z_j}$ being a Carleson measure (or any of the
equivalent conditions of \cite{DS02}). This holds for $H^p$ for all $0 <
p \le \infty$. The proof of this will appear elsewhere.
\end{remark}

We now turn to a few examples of interpolation schemes.

\begin{example}\label{ex:1}
Simple interpolation can be cast as an interpolation scheme $\I$ where
each cluster is a single point $\{ z_k \}$. If we choose $G_k = D(z_k,
R)$ for some $R < 1$, then each $E_k$ is one-dimensional and the
quotient map from $A^p(G_k)$ to $E_k$ can be identified with evaluation
at $z_k$. If $g \in A^p(G_k)$ and $g(z_k) = w_k$, then
\begin{equation*}
| w_k |^p \le \frac{C_R}{|G_k|}\int_{G_k} |g|^p \,dA,
\end{equation*}
where $C_R$ is an constant depending only on $R$. On taking the infimum,
we see $|w_k|^p|G_k| \le C_R\| w_k \|_{E_k}^p$. On the other hand,
considering the constant function identically equal to $w_k$ on $G_k$
shows that $\| w_k \|_{E_k}^p \le |w_k|^p|G_k|$. Since $|G_k|$ behaves
like $(1 - |z_k|^2)^2$, the norm on the sequence space $X^p_\I$ is
equivalent to the norm $\| w \|_{p,\Z}^p = \sum |w_k|^p(1 - |z|^2)^2$.
This scheme is therefore equivalent to simple interpolation.

The main theorem here then contains K.~Seip's density criterion for
simple interpolation.
\end{example}

\begin{example}
For multiple interpolation of order $n$, define $\I$ so that each
cluster is a single point $z_k$ repeated $n$ times and the domains $G_k$
are $D(z_k,R)$ as before. Then $E_k$ is isomorphic to $\IC^n$. If $f\in
A^p(G_k)$ with values $f^{(j)}(z_k) = w^{(j)}_k$ for $0 < j < n-1$, then it
is well known that
\begin{equation*}
    |w_k^{(j)}(1 - |z_k|^2)^j|^p \le \frac{C}{|G_k|}\int_{G_k}
        |g|^p \,dA,
\end{equation*}
where $C$ depends at most on $p$, $j$ and $R$. Summing, and taking the
infimum over all $f$ with the same values $w_k = (w_k^{(0)}, \dots,
w_k^{(n-1)})$ gives
\begin{equation*}
  \sum_{j=0}^{n-1} |w_k^{(j)}|^p (1 - |z_k|^2)^{pj + 2} \le C\| w_k
    \|_{E_k}.
\end{equation*}
Taking $f = \sum_{j=0}^{n-1} w^{(j)}_k(z-z_k)^j$ easily gives a reverse
inequality, so this scheme is equivalent to multiple interpolation.
Note, however, that we now can allow the multiplicities to vary. By
Theorem~\ref{thm:boundeddensity}, interpolating sequences for this
scheme must have an upper bound on the multiplicities.

The main theorem here essentially contains the equivalent of Krosky and
Schuster's density criterion for multiple interpolation.
\end{example}

\begin{example}
Finally, we present the simplest example of a sequence which has
distinct points, but is not uniformly discrete. Let $\{ a_k \st k\ge 1
\}$ and $\{ b_k \st k \ge 1 \}$ be two sequences without limit points in
$\ID$ and satisfying $a_k \ne b_k$ and $\eps_k = \psi(a_k, b_k)$ is
bounded away from $1$. Select a radius $R \in (0,1)$ such that $\sup_k
\eps_k < R$ and choose domains $G_k = D(a_k,R)$ for our interpolation
scheme $\I$. Let the clusters be $\Z_k = \{ a_k,b_k \}$, and note that
they satisfy $(\Z_k)_\eps \subset G_k$ for some $\eps > 0$. $E_k$ can be
identified with $\IC^2$ where the equivalence class of a function $f$ is
mapped to $(f(a_k), f(b_k))$.

Let $w_k = (u_k, v_k) \in E_k$. By inequalities we have already seen many
times, if $f\in A^p(G_k)$ satisfies $f(a_k) = u_k$ and $f(b_k) = v_k$,
then
\begin{equation*}
  |u_k|^p + |(v_k-u_k)/\eps_k|^p \le \frac{C}{|G_k|} \int_{G_k}
    |f|^p \,dA,
\end{equation*}
which gives us $\left( |u_k|^p + |(v_k-u_k)/\eps_k|^p \right)|G_k| \le
\| w_k \|_{E_k}$. Taking the equivalence class representative
\begin{equation*}
    f(z) = u_k  + (v_k - u_k)\frac{z - a_k}{1 - \bar a_k
    z}\biggm/\frac{b_k - a_k}{1 - \bar a_k b_k}
\end{equation*}
gives a reverse inequality. Therefore, in the sequence space $X^p_\I$,
the norm $\| w \|_{X^p_\I}^p$ is equivalent to $\sum \left( |u_k|^p
+ |(u_k - v_k)/\eps_k|^p \right)(1 - |a_k|^2)^2$. This gives the
following result:
\begin{proposition}
  Given $\Z = \{ a_1, b_1, a_2, b_2,\dots \}$ let the sequence
  $\psi(a_k,b_k)$ be bounded away from 1. Then the following are
  equivalent.
  \begin{enumerate}
    \item For every sequence $((u_k,v_k)\st k\ge 1)$ satisfying
    \begin{equation*}
        \sum_{k = 1}^\infty \left( |u_k|^p + \left| \frac{v_k -
        u_k}{\psi(a_k,b_k)} \right|^p \right) (1 - |a_k|^2)^2
        < \infty
    \end{equation*}
    there is a function $f\in A^p$ satisfying $f(a_k) = u_k$ and $f(b_k)
    = v_k$.
    \item There exists $\delta>0$ such that the clusters $\Z_k =
        \{a_k,b_k\}$ satisfy $\psi(\Z_k,\Z_j) > \delta$ for all $j \ne
        k$ and $D^+(\Z) < 1/p$.
  \end{enumerate}
\end{proposition}

If the $\psi(a_k,b_k)$ are bounded away from both $0$ and $1$, then this
is just simple interpolation and the norm is equivalent to the one in
example~\ref{ex:1}. We get something new if $\psi(a_k,b_k)$ is not
bounded away from $0$.
\end{example}

\section{Addendum}\label{sec:addendum}

\subsection{Sampling}

A student of mine, N.~H.~Foster, has recently completed a dissertation
in which he proved the complementary result for sampling. That is, given
an admissible scheme $\I$ with domains $G_k$ and clusters $\Z_k$, let
$\Phi(f)$ denote the sequence of cosets of $f|_{G_k}$ for any $f\in A^p$.
Call $\Z = \Union_k \Z_k$ \emph{a sampling sequence for $A^p$ relative
to $\I$} if there exists a constant $C$ such that for all $f\in A^p$
$$
  \frac{1}{C}\| \Phi(f) \|_{X^p_\I} \le \| f \|_{p} \le C\| \Phi(f)
    \|_{X^p_\I}\,.
$$
The definition of the $X^p_\I$ norm makes the left side inequality
automatic, so the main point is the second inequality. His result is
that $\Z$ is sampling if and only if a companion density $D^{-}(\Z)$
(called the \emph{lower uniform density}) is greater than $1/p$.

His result is likely true for the weighted spaces $A^{p,\alpha}$ as
well, with $D^{-}(\Z) > (\alpha+1)/p$.

\subsection{O-interpolation}

I was recently made aware of a paper by S. Ostrovsky \cite{Ost10} and
one by A. Schuster and T. Wertz \cite{SW13}. The first is set in spaces
of entire functions with certain exponential weights (generalizations of
the Fock spaces), the second is set in the unit disk, with weights
analogous to those of \cite{Ost10}. Both papers limit the results to
$p=2$, but the weights considered are more general than those considered
here.

I have extended the general interpolation results obtained in this
paper to the weights considered in \cite{SW13}, but those results will
appear elsewhere. Those weights have the form
\begin{equation*}
  \frac{e^{-\phi(z)}}{1 - |z|^2}
\end{equation*}
where $\phi$ is $C^2$ on $\ID$ and satisfies $0 < m \le \invL \phi(z)
\le M$ for positive constants $m$ and $M$. The standard weights
$(1-|z|^2)^\alpha$, $\alpha > -1$, correspond to
\begin{equation*}
   \phi(z) = (\alpha + 1) \log \left( \frac{1}{1 - |z|^2} \right)
\end{equation*}
which satisfy $\invL \phi = \alpha + 1$.

I will limit my remarks here to showing that the results in \cite{SW13},
when applied to the standard weighted Bergman spaces $A^{p,\alpha}$,
follow from the results here. In fact, a version for $p \ne 2$ also
follows.

Suppose $\Z$ is a sequence of distinct points in $\ID$. For each point
$\gamma\in \Z$ let $n_\gamma$ be the number of points in $\Z\intersect
D(\gamma,1/2)$ and let $\delta_\gamma = \inf\{ \psi(\gamma,\beta):
\beta\in \Z, \beta \ne \gamma \}$ be the pseudohyperbolic distance to
the nearest other point of $\Z$. The following is the main theorem of
\cite{SW13}, restricted to the standard weights but extended to all $p >
0$.

\begin{theorem}
  If $p > 0$ and $\Z$ is a sequence of distinct points satisfying
  $D^+(\Z) < (\alpha+1)/p$, then for every sequence $\{ c_\gamma :
  \gamma \in Z \}$ satisfying
\begin{equation}\label{eq:specialweights}
  \sum_{\gamma \in \Z} |c_\gamma|^p \frac{(1 - |\gamma|^2)^{\alpha+2}}
  {\delta_\gamma^{pn_\gamma}} < \infty
\end{equation}
  there is a function $f \in A^{p,\alpha}$ satisfying $f(\gamma) =
  c_\gamma$ for all $\gamma\in\Z$.
\end{theorem}

 The paper \cite{SW13} calls this mode of
interpolation, \emph{O-interpolation}.

The sum in \eqref{eq:specialweights} defines a norm for a certain
weighted sequence space. If we can set up an admissible interpolation
scheme where the sequence space $X^p_{\I}$ essentially contains this
space, then the implication \ref{UUD} $\Rightarrow$ \ref{GIS1} of
Theorem~\ref{thm:realmain} proves the above theorem.

\begin{proof}
For simplicity of exposition, let us assume $\alpha = 0$ so that we have
the ordinary unweighted Bergman spaces.

Suppose $D^+(\Z) < 1/p$. Then certainly $\Z$ has bounded
density, and so there exists an admissible scheme for $\Z$. Let $\Z_k$ be
the clusters and $G_k$ the domains of this scheme.

We can get an estimate on the norm of an equivalence class $w_k \in
E_k$: it is less than the norm of any function analytic on $G_k$
belonging to that equivalence class. Since $\Z$ consists of distinct
points, $w_k$ is determined by the values of such a function at the points
in $\Z_k$, one such function, which has the values $c_\gamma$ is
\begin{equation*}
  f_k(z) = \sum_{\gamma\in \Z_k} c_\gamma \prod_{\substack{\beta\in \Z_k\\
  \beta\ne\gamma}} \frac{M_\beta(z)} {M_{\beta}(\gamma)}\,
\end{equation*}
(recalling that $M_\beta(z)$ is the Moebius transformation $(\beta-z)/(1
- \bar\beta z)$). Note that there is an upper bound $B$ on the
cardinality of $\Z_k$. Therefore each product in this formula contains at
most $B$ factors. We wish to estimate the size of this product. If we
split it into those factors with $\beta \in D(\gamma,1/2)$ ($n_\gamma$
factors) and the rest (at most $B$ factors), we can estimate
\begin{equation*}
  \left| \prod_{\beta} \frac{M_\beta(z)} {M_{\beta}(\gamma)}
  \right| \le \frac{2^B}{\delta_\gamma^{n_\gamma}}
\end{equation*}
Thus
\begin{equation*}
  \| w_k \|^p_{E_k} \le C \sum_{\gamma \in \Z_k} |c_\gamma|^p
    \frac{(1-|\gamma|^2)^2} {\delta_\gamma^{pn_\gamma}}
\end{equation*}
Thus, if a sequence $c_\gamma$ satisfies the finiteness condition of the
theorem, the $(w_k)$ determined by these $f_k$ belongs to $X^p_\I$ and
so a function $f \in A^p$ exists that agrees with each $f_k$ on $\Z_k$,
that is, $f(\gamma) = c_\gamma$, as required.
\end{proof}

%%%%%%%%%%%%%%%%%%%%%%%%%%%%%
%\bibliographystyle{amsplain}
%\bibliography{luecking}
%%%%%%%%%%%%%%%%%%%%%%%%%%%%%

\newcommand{\noopsort}[1]{} \providecommand{\bysame}{\leavevmode\hbox
to3em{\hrulefill}\thinspace}
\providecommand{\MR}{\relax\ifhmode\unskip\space\fi MR }
\providecommand{\MRhref}[2]{%
  \href{http://www.ams.org/mathscinet-getitem?mr=#1}{#2}
}
\providecommand{\href}[2]{#2}

\end{document}